\documentclass[12pt]{amsart}

\usepackage{amssymb, amsmath, amsfonts, cite, tikz, graphicx}
\usepackage[colorlinks]{hyperref}
\usepackage[capitalize]{cleveref}
\usepackage[a4paper]{geometry}

\newtheorem{thm}{Theorem}[section]

\newtheorem{lem}[thm]{Lemma}

\theoremstyle{definition}
\newtheorem{case}{Case}[thm]

\makeatletter \@addtoreset{equation}{section} \makeatother

\def\fl#1{\lfloor{#1}\rfloor}
\def\bfl#1{\bigl\lfloor\,{#1}\,\bigr\rfloor}
\def\bg#1{\bigl({#1}\bigr)}

\def\bgg#1{\biggl({#1}\biggr)}
\def\C{\mathbb{C}}
\def\R{\mathbb{R}}
\def\rmand{\quad\hbox{ and }\quad}

\DeclareMathOperator\sgn{\mathrm{sgn}}

\crefname{case}{Case}{Cases}
\crefname{case}{Case}{Cases}
\creflabelformat{case}{~\upshape#2#1#3}
\crefname{eg}{Eg.}{Egs.}
\Crefname{eg}{Example}{Examples}
\creflabelformat{eg}{~\upshape#2#1#3}
\crefname{ineq}{Ineq.}{Ineqs.}
\Crefname{ineq}{Inequality}{Inequalities}
\creflabelformat{ineq}{~\upshape(#2#1#3)}
\crefname{lem}{Lemma}{Lemmas}
\Crefname{lem}{Lemma}{Lemmas}
\creflabelformat{lem}{~\upshape#2#1#3}
\crefname{rec}{Recurrence}{Recurrences}
\Crefname{rec}{Recurrence}{Recurrences}
\creflabelformat{rec}{~\upshape(#2#1#3)}
\crefname{sec}{\S\!}{\S\!}
\Crefname{sec}{Section}{Sections}
\creflabelformat{sec}{~\upshape#2#1#3}
\crefname{ssec}{\S\!}{\S\!}
\Crefname{ssec}{Subsection}{Subsections}
\creflabelformat{ssec}{~\upshape#2#1#3}
\crefname{thm}{Theorem}{Theorems}
\Crefname{thm}{Theorem}{Theorems}
\creflabelformat{thm}{~\upshape#2#1#3}

\crefrangeformat{equation}{Eqs.\!~(#3#1#4) ---~(#5#2#6)}
\crefrangeformat{ineq}{Ineqs.\!~(#3#1#4) ---~(#5#2#6)}
\crefrangeformat{thm}{Theorems\!~#3#1#4 ---~#5#2#6}
\crefrangeformat{figure}{Figs.\!~#3#1#4 ---~#5#2#6}

\numberwithin{equation}{section}
\numberwithin{figure}{section}

\usepackage{xpatch}
\xpatchcmd{\itemize}{\def\makelabel}{\setlength{\itemsep}{5pt}\def\makelabel}{}{}

\allowbreak
\allowdisplaybreaks

\def\L{\mathcal{L}}

\begin{document}

\title[]{{Root geometry of polynomial sequences III:\\[5pt]
Type $(1,1)$ with positive coefficients}}

\author[D.G.L. Wang]{David G.L. Wang$^\dag$$^\ddag$}
\address{
$^\dag$School of Mathematics and Statistics, Beijing Institute of Technology, 102488 Beijing, P. R. China\\
$^\ddag$Beijing Key Laboratory on MCAACI, Beijing Institute of Technology, 102488 Beijing, P. R. China}
\email{david.combin@gmail.com}

\author[J.J.R. Zhang]{Jerry J.R. Zhang$^\dag$}
\address{
$^\dag$School of Mathematics and Statistics, Beijing Institute of Technology, 102488 Beijing, P. R. China}
\email{jrzhang.combin@gamil.com}

\begin{abstract}   
In this paper, we study the root distribution of some univariate polynomials $W_n(z)$ satisfying a recurrence of order two with  linear polynomial coefficients over positive numbers. We discover a sufficient and necessary condition for the overall real-rootedness of all the polynomials, in terms of the polynomial coefficients of the recurrence. Moreover, in the real-rooted case, we find the set of limits of zeros, which turns out to be the union of a closed interval and one or two isolated points; when non-real-rooted polynomial exists, we present a sufficient condition under which every polynomial with $n$ large has a real zero.
\end{abstract}
%03D20: Recursive functions and relations, subrecursive hierarchies
%26C10: Polynomials: location of zeros
%30C15: Zeros of polynomials, rational functions, and other analytic functions
%(e.g. zeros of functions with bounded Dirichlet integral) 
\maketitle
\parskip 9pt

\section{Introduction}
The root distribution of a single polynomial is a long-standing topic all along the history of mathematics; 
see Rahman and Schmeisser's book \cite{RS02B}. 
Motived by the LCGD conjecture from topological graph theory,
Gross, Mansour, Tucker and the first author~\cite{GMTW16-01,GMTW16-10}
studied the root distribution of polynomials satisfying some recurrences 
of order two, with one of the polynomial coefficients in the recurrence linear and the other constant. 
The generating function of such polynomials $W_n(z)$ is
\begin{equation}\label{gf:W}
\sum_{n\ge0}W_n(z)t^n=\frac{1+(z-A(z))t}{1-A(z)t-B(z)t^2},
\end{equation}
where $A(z)$ and $B(z)$ are the polynomial coefficients in the recurrence.
Orthogonal polynomials and quasi-orthogonal polynomials have closed relations with
such recurrences; 
see Andrews, Richard and Ranjan~\cite{ARR99B} and Brezinski, Driver and Redivo-Zaglia~\cite{BDR04}.
A general study for common zeros as a particular case of root distribution for 
polynomials defined by recurrences of order two can be found in \cite{JW17X}.

In the study of root distribution, 
both the real-rootedness and the limiting distribution of zeros of the polynomials 
received much attention. 
Stanley on his website \cite{StaW}
provides some figures for the root distribution of some polynomials in a sequence 
arising from combinatorics. 
The significance of real-rootedness of polynomials can be found from~\cite[\S 4]{Sta00}.
See also~\cite{BM06,BG07,BG08,GHR09}.

In this paper, we continue the study of root distribution of
polynomials generated by \cref{gf:W}, where both $A(z)$ and $B(z)$ are linear and over $\R^+$.
It turns out that the linearities bring much richer root geometry.
We find a sufficient and necessary condition for the real-rootedness 
of such polynomials,
and determine the set of limits of zeros when the polynomials are real-rooted,
by using the characterization of the limits given by Beraha, Kahane and Weiss~\cite{BKW75,BKW78}.
A study on the root distribution of polynomials generated by \cref{gf:W} with
another sort of linear polynomial coefficients $A(z)$ and $B(z)$
can be found in \cite{WZ17X}, in which we introduced the notion of piecewise interlacing zeros
to establish the real-rootedness subject to some technical conditions,
but failed to find a sufficient and necessary condition for the overall real-rootedness.
Polynomials generated by the function
\[
\sum_{n\ge0}W_n(z)t^n=\frac{1}{1-A(z)t-B(z)t^2}
\]
has been investigated in \cite{Tra14}, 
in which Tran found an algebraic curve containing the zeros of all polynomials with large subscripts.

This paper is organised as follows. In \cref{sec:main}, we state the main results \cref{thm:cri:RR,thm:lz,thm:xA>xB:1R}. The first two results are shown in \cref{sec:xA<=xB}
while the last result is proved in \cref{sec:xA>xB}.

\section{Main results}\label[sec]{sec:main}

One main result of this paper is a characterisation for the real-rootedness of 
some polynomials defined by a recurrence.

\begin{thm}\label{thm:cri:RR}
Let $\{W_n(z)\}_n$ be the polynomial sequence satisfying the recurrence
\begin{equation}\label[rec]{rec11}
W_n(z) = (az+b)W_{n-1}(z)+(cz+d)W_{n-2}(z). 
\end{equation}
with $W_0(z)=1$ and $W_1(z)=z$, where $a,b,c,d>0$.
Then every polynomial $W_n(z)$ is real-rooted if and only if $ad\le bc$.
\end{thm}

The sufficiency part of \cref{thm:cri:RR} 
will be established by proving that consecutive polynomials
have piecewise interlacing zeros by induction,
while the necessity part is shown with the aid of the limits of zeros;
see \cref{sec:sufficiency,sec:necessity} for details respectively.

Another main result is about the limit of zeros of such polynomials 
for the real-rooted case $ad<bc$.
A complex number $x$ is said to be a {\em limit of zeros} of a sequence $\{W_n(z)\}_n$ of polynomials
if there is a zero $z_n$ of $W_n(z)$ for each $n$ 
such that $\lim_{n\to\infty}z_n=x$. 
Such a limit $x$ is said to be {\em non-isolated} 
if $x$ is a limit point of the set of all limits of zeros.

In this paper, we consider polynomials $W_n(z)$ satisfying \cref{rec11} with $a,b,c,d>0$.
Let $\Delta_\Delta=-a^2d+abc+c^2$ and $\Delta_g=(b+c)^2+4d(1-a)$.
Let
\begin{align*}
x_\Delta^{\pm}&=\frac{-ab-2c\pm2\sqrt{\Delta_\Delta}}{a^2}\rmand\\[5pt]
x_g^\pm&=\begin{cases}
\displaystyle
\frac{b+c}{2(1-a)}\pm\frac{\sqrt{\Delta_g}}{2|1-a|},&\text{if $a\ne1$},\\[8pt]
\displaystyle -{d\over b+c},&\text{if $a=1$}.
\end{cases}
\end{align*}
Suppose that $ad\le bc$.
It is direct to check that $\Delta_\Delta>0$ and $\Delta_g>0$. 
Then $x_\Delta^\pm,x_g^\pm\in\R$ and $x_\Delta^-<-b/a<x_\Delta^+$.

\begin{thm}\label{thm:lz}
Let $\{W_n(z)\}_n$ be the polynomial sequence satisfying \cref{rec11} with $W_0(z)=1$ and $W_1(z)=z$, where $a,b,c,d>0$.
Then the points $-b/a$ and $x_\Delta^\pm$ 
are non-isolated limits of zeros of $\{W_n(z)\}_n$.
Moreover, when $ad<bc$, the set of limits of zeros is
\[
\mathcal{L}^*=\begin{cases}
[x_\Delta^-,\,x_\Delta^+]\cup\{x_g^-\},
&\text{if $0<a\le 1$};\\[4pt]
[x_\Delta^-,\,x_\Delta^+]\cup\{x_g^+\},
&\text{if either $1<a\le 2$, or $a>2$ and $\Delta_\Delta\le \Delta_g$};\\[4pt]
[x_\Delta^-,\,x_\Delta^+]\cup\{x_g^\pm\},&\text{if $a>2$ and $\Delta_\Delta>\Delta_g$}.
\end{cases}
\]
\end{thm}

When $ad>bc$, even the polynomial $W_2(z)=az^2+(b+c)z+d$ may have no real zeros.
We obtain a sufficient condition for the existence of one real zero for large $n$.

\begin{thm}\label{thm:xA>xB:1R}
Let $\{W_n(z)\}_n$ be the polynomial sequence satisfying \cref{rec11} 
with $W_0(z)=1$ and $W_1(z)=z$, where $a,b,c,d>0$ and $ad>bc$. 
Suppose that at least one of the numbers $(2-a)x_\Delta^\pm-b$ is positive.
Then there exists $N$ such that the polynomial $W_n(z)$ has at least one real zero for all $n>N$.
\end{thm}

\section{Proof of \cref{thm:cri:RR,thm:lz}}\label[sec]{sec:xA<=xB}
To the end of this paper, we use the notation $\{W_n(z)\}_n$ 
to denote a polynomial sequence satisfying \cref{rec11}
with $W_0(z)=1$ and $W_1(z)=z$, where $a,b,c,d>0$.
We employ the notations
\begin{align*}
\Delta(z)&=A(z)^2+4B(z)
=a^2z^2+(2ab+4c)z+(b^2+4d),\\[3pt]
\alpha_\pm(z)&=\frac{\sqrt{\Delta(z)}\pm h(z)}{2\sqrt{\Delta(z)}},\rmand\\[3pt]
g(z)&=-\alpha_+(z)\alpha_-(z)\Delta(z)
=\frac{h^2(z)-\Delta(z)}{4}=(1-a)z^2-(b+c)z-d,
\end{align*}
where $h(z)=(2-a)z-b$; see \cite{GMTW16-01,GMTW16-10}.
Then the numbers $x_g^\pm$ and $x_\Delta^\pm$ 
are zeros of the polynomials $\Delta(z)$ and $g(z)$ respectively,
and 
\begin{align}
W_n(x_g^\pm)&=(x_g^\pm)^n,\label{W:xg}\\[3pt]
W_n(x_\Delta^\pm)
&=\frac{A(x_\Delta^\pm)+nh(x_\Delta^\pm)}{2}
\cdot\bgg{\frac{A(x_\Delta^\pm)}{2}}^{n-1},\label{W:xd}
\end{align}
For convenience, we denote $x_A=-b/a$ and $x_B=-d/c$.
We use the notation $\sgn(\cdot)$ to denote the sign function for real numbers.
Let
\begin{equation*}
n^\pm=-\frac{A(x_\Delta^\pm)}{h(x_\Delta^\pm)},
\qquad\text{if $h(x_\Delta^\pm)\ne0$.}
\end{equation*}

\begin{lem}\label{++++:lem:basic}
We have
\begin{align*}
\sgn\bg{W_n(x_A)}
&=\begin{cases}
(-1)^{\lceil n/2\rceil},&\text{if $x_A<x_B$},\\
0,&\text{if $x_A=x_B$},\\
(-1)^n,&\text{if $x_A>x_B$}.
\end{cases}
\end{align*}
If $\Delta_g\ge0$, then $W_n(x_g^-)(-1)^n>0$. If $\Delta_\Delta\ge0$, then 
\[
\sgn\bg{W_n(x_\Delta^-)}=\begin{cases}
(-1)^n,&\text{if $h(x_\Delta^-)\le0$ or $n<n^-$};\\[4pt]
0,&\text{if $n=n^-$};\\[4pt]
(-1)^{n+1},&\text{otherwise}.
\end{cases}
\]
Moreover, we have $h(x_\Delta^\pm)<0$ if $a\le2$. 
\end{lem}
\begin{proof}
The sign of $W_n(x_A)$ can be shown directly from \cref{rec11},
that of $W_n(x_g^-)$ by \cref{W:xg},
and that of $W_n(x_\Delta^-)$ by \cref{W:xd}. 
Note that $x_\Delta^-\le x_\Delta^+<0$. 
Hence $h(x_\Delta^\pm)=(2-a)x_\Delta^\pm-b<0$ if $a\le2$.
\end{proof}
 
In \cref{sec:sufficiency}, we show the sufficiency part of \cref{thm:cri:RR}.
\Cref{thm:lz} is  established in \cref{sec:lz},
and used in proving the necessity part of \cref{thm:cri:RR} in \cref{sec:necessity}.

\subsection{The sufficiency part of \cref{thm:cri:RR}.}\label[ssec]{sec:sufficiency}
For convenience, we denote
\[
u=\begin{cases}
x_\Delta^-,&\text{if $a\le 2$},\\[4pt]
x_g^-,&\text{if $a>2$},
\end{cases}
\rmand
v=\begin{cases}
x_g^-,&\text{if $a\le 1$},\\[4pt]
x_g^+,&\text{if $a>1$}. 
\end{cases}
\]

\begin{lem}\label{++++:lem:basic:xA<xB}
Suppose that $x_A\le x_B$.
Then $W_n(x_\Delta^+)<0$,
$W_n(u)(-1)^n>0$ and $W_n(v)(-1)^n>0$.
When $a>1$, we have $x_g^-\le x_\Delta^-$, with the equality holds if and only if $\Delta_\Delta=\Delta_g$.
Moreover, we have the following. 
\begin{itemize}
\item
If $x_A<x_B$, then $u\le x_\Delta^-<x_A<x_\Delta^+<v<0$.
\item
If $x_A=x_B$, then $u<x_A=x_\Delta^+<v<0$. 
\end{itemize}
\end{lem}

\begin{proof}
It is clear that 
\[
x_A\le x_B\iff ad\le bc.
\]
Suppose that $x_A\le x_B$. It is direct to check that $\Delta_\Delta,\Delta_g\ge0$, 
which implies $x_\Delta^\pm,x_g^\pm\in\R$. From definition, one may verify $x_\Delta^-< x_A\le x_\Delta^+<0$.
When $x_A<x_B$, we have $\Delta(x_A)=4B(x_A)<0$, which implies 
$x_A<x_\Delta^+$.
When $x_A=x_B$, we have $x_\Delta^+=x_A$ from definition. 
It follows that $A(x_\Delta^-)\le 0\le A(x_\Delta^+)$, and that 
\begin{align}
h(x_\Delta^\pm)=(2-a)x_\Delta^\pm-b\le-b<0,&\qquad\text{if $a\le 2$},
\label[ineq]{pf:h:xd:a<=2}\\
h(x_\Delta^+)=(2-a)x_\Delta^+-b\le (2-a)x_A-b=2x_A<0,&\qquad\text{if $a>2$}.
\label[ineq]{pf:h:xd+:a>2}
\end{align}
Therefore, we have
\begin{align*}
&A(x_\Delta^\pm)+nh(x_\Delta^\pm)\le A(x_\Delta^\pm)+h(x_\Delta^\pm)=2x_\Delta^\pm<0,\qquad\text{if $a\le 2$},\rmand\\[3pt]
&A(x_\Delta^+)+nh(x_\Delta^+)\le A(x_\Delta^+)+h(x_\Delta^+)=2x_\Delta^+<0,\qquad\text{if $a>2$}.
\end{align*}
By \cref{W:xd}, we find $W_n(x_\Delta^+)<0$, and 
\[
W_n(x_\Delta^-)(-1)^n>0\qquad\text{if $a\le 2$}.
\]
On the other hand, from definition, one may verify that 
\begin{align*}
x_g^-<0<x_g^+,&\qquad\text{if $a<1$}\rmand\\
x_g^-\le x_g^+<0,&\qquad\text{if $a\ge1$}.
\end{align*}
By \cref{W:xg} and the definition of $v$, we find $W_n(v)(-1)^n>0$.

From definition, we have
$g(x_\Delta^\pm)=h^2(x_\Delta^\pm)/4$.
By using \cref{pf:h:xd:a<=2,pf:h:xd+:a>2}, we find 
\begin{itemize}
\item
if $a\le 2$, then $g(x_\Delta^\pm)>0$; and
\item
if $a>2$, then $g(x_\Delta^+)>0$ and $g(x_\Delta^-)\ge0$.
\end{itemize}
%\begin{align*}
%g(x_\Delta^\pm)>0,&\qquad\text{if $a\le2$};\\[3pt]
%g(x_\Delta^+)>0,&\qquad\text{if $a>2$};\\[3pt]
%g(x_\Delta^-)\ge0,&\qquad\text{if $a>2$}.
%\end{align*}
When $a>1$, the function $g(z)$ is a downward parabola.
Then the points~$x_\Delta^\pm$ lie inside the closed interval $[x_g^-,\,x_g^+]$. 
When $a<1$, $g(z)$ is an upward parabola and
$x_\Delta^\pm$ lie outside the open interval $(x_g^-,\,x_g^+)$. 
Since $x_\Delta^+<0<x_g^+$,
both~$x_\Delta^\pm$ lie to the left of the interval. 
When $a=1$,
the function $g(z)$ is a decreasing line and $x_\Delta^+\le x_g^\pm$.
In summary, we have
\begin{align*}
x_\Delta^-< x_\Delta^+<x_g^-<0<x_g^+,&\qquad\text{if $a<1$},\\
x_\Delta^-<x_\Delta^+<x_g^-=x_g^+<0,&\qquad\text{if $a=1$},\\
x_g^-\le x_\Delta^-< x_\Delta^+<x_g^+<0,&\qquad\text{if $a>1$}.
\end{align*}
The equality $x_g^-=x_\Delta^-$ holds if and only if $g(x_\Delta^-)=0$, that is,
$h(x_\Delta^-)=0$. It is routine to check that
\[
h(x_\Delta^-)=\frac{2(\Delta_g-\Delta_\Delta)}{-ab+ac-2c-(a-2)\sqrt{\Delta_\Delta}}.
\]
Hence $x_g^-=x_\Delta^-$ if and only if $\Delta_\Delta=\Delta_g$. 
All the other desired inequalities are verified in the above proof. 
This completes the proof.
\end{proof}

We remark that $\Delta_\Delta=\Delta_g$ implies that $a>2$.

Let $X,Y\subset\R$ such that $|X|-|Y|\in\{0,1\}$.
We say that {\em $X$ interlaces $Y$ from the left}, 
if the elements $x_i$ of $X$ and the elements $y_j$ of $Y$ can be arranged so that 
$x_1\le y_1\le x_2\le y_2\le\cdots$,
and that {\em $X$ strictly interlaces $Y$ from the left} 
if no equality holds in the ordering.
We say {\em $X$ interlaces $Y$ from the right}
if the elements can be arranged as
$x_1\ge y_1\ge x_2\ge y_2\ge\cdots$,
and {\em $X$ strictly interlaces~$Y$ from the right} if $x_1>y_1>x_2>y_2>\cdots$.
For any set $S$ and any interval $J$, we denote $S^J=S\cap J$ for short notation.
 
It is clear that the leading coefficient of $W_n(z)$ is $a^{n-1}z^n$,
and the polynomials $W_n(z)$ satisfy another recurrence 
\begin{equation}\label[rec]{rec2}
W_n(z)=\bg{A^2(z)+2B(z)}W_{n-2}(z)-B^2(z)W_{n-4}(z),\qquad\text{for $n\ge4$}.
\end{equation}
Using the idea of piecewise interlacing zeros introduced in \cite{WZ17X}, 
one may derive the root distribution of the polynomials $W_n(z)$ when $x_A<x_B$;
see \cref{++++:thm:xA<xB}, and \cref{fig1:++++:thm:xA<xB,fig2:++++:thm:xA<xB}
for illustration.

\begin{thm}\label{++++:thm:xA<xB}
Suppose that $x_A<x_B$. Then every polynomial $W_n(z)$ is real-rooted. 
Denote by $R_n$ the zero set of~$W_n(z)$.
Let
\[
J_1=(u,\,x_A),\quad
J_2=(x_A,\,x_\Delta^+),\quad
J_3=(x_\Delta^+,\,v),\quad
J_4=(v,\,0]. 
\]
Then $|R_n^{J_1}|=\fl{n/2}$, $|R_n^{J_2}|=\fl{(n-1)/2}$, 
and $|R_{2n}^{J_3}|=|R_{2n-1}^{J_4}|=1$.
Moreover, for any $k=1,2$ and any $j=1,2$,
the set $R_n^{J_j}$ strictly interlaces $R_{n-k}^{J_j}$, from the left if $n$ is odd,
and from the right if $n$ is even. 
\end{thm}

\begin{figure}[htbp]
\begin{tikzpicture}[scale=1.2]
\draw[->] (0,0) -- (10,0) coordinate (x axis);
\foreach \x in {1,3,...,9} 
 \draw[very thick] (\x cm,2pt) -- (\x cm,-2pt);
\draw[very thick] (1,0) -- node[below=5pt] {$J_1$} (3,0);
\draw[very thick] (3,0) -- node[below=5pt] {$J_2$} (5,0);
\draw(5,0) -- node[below=5pt] {$J_3$} (7,0);
\draw(7,0) -- node[below=5pt] {$J_4$} (9,0);

\draw (1,0) node[below=7pt] {$u$};
\draw (3,0) node[below=7pt] {$x_A$};
\draw (5,0) node[below=3pt] {$x_\Delta^+$};
\draw (7,0) node[below=7pt] {$v$};
\draw (9,0) node[below=7pt] {$0$};

\draw (1,0) -- node[above=5pt] {$\bfl{\frac{n}{2}}$} (3,0);
\draw (3,0) -- node[above=5pt] {$\bfl{\frac{n-1}{2}}$} (5,0);
\draw (7,0) -- node[above left = 5pt] {$1$} (9,0);
\end{tikzpicture}
\caption{Illustration of the root distribution of the polynomials $W_{2n-1}(z)$ 
in \cref{++++:thm:xA<xB}.}\label{fig1:++++:thm:xA<xB}
\end{figure}
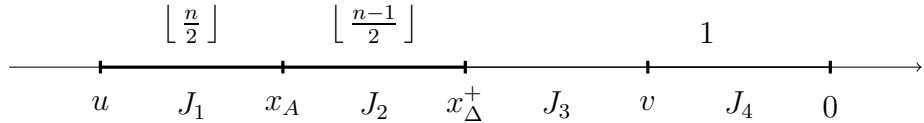

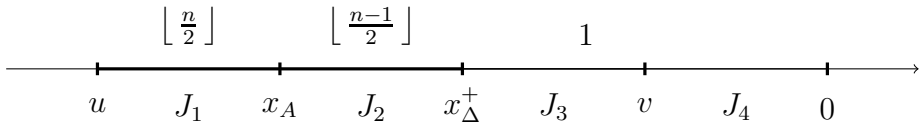
\begin{figure}[htbp]
\begin{tikzpicture}[scale=1.2]
\draw[->] (0,0) -- (10,0) coordinate (x axis);
\foreach \x in {1,3,...,9} 
 \draw[very thick] (\x cm,2pt) -- (\x cm,-2pt);
\draw[very thick] (1,0) -- node[below=5pt] {$J_1$} (3,0);
\draw[very thick] (3,0) -- node[below=5pt] {$J_2$} (5,0);
\draw(5,0) -- node[below=5pt] {$J_3$} (7,0);
\draw(7,0) -- node[below=5pt] {$J_4$} (9,0);

\draw (1,0) node[below=7pt] {$u$};
\draw (3,0) node[below=7pt] {$x_A$};
\draw (5,0) node[below=3pt] {$x_\Delta^+$};
\draw (7,0) node[below=7pt] {$v$};
\draw (9,0) node[below=7pt] {$0$};

\draw (1,0) -- node[above=5pt] {$\bfl{\frac{n}{2}}$} (3,0);
\draw (3,0) -- node[above=5pt] {$\bfl{\frac{n-1}{2}}$} (5,0);
\draw (5,0) -- node[above right = 5pt] {$1$} (7,0);
\end{tikzpicture}
\caption{Illustration of the root distribution of the polynomials $W_{2n}(z)$ 
in \cref{++++:thm:xA<xB}.}\label{fig2:++++:thm:xA<xB}
\end{figure}

\begin{proof}
Suppose that $x_A<x_B$. 
By \cref{++++:lem:basic:xA<xB}, we have 
\begin{equation}\label[ineq]{++++:W:xd1:xA<=xB}
W_n(u)(-1)^n>0.
\end{equation}
By the intermediate value theorem and with the aids of 
\cref{++++:W:xd1:xA<=xB,++++:lem:basic},
it is routine to check that the polynomial $W_n(z)$ ($n=2,3,4$)
has zeros $\xi_{n,1}<\cdots<\xi_{n,n}$ such that 
\[
u<
\xi_{4,1}<\xi_{3,1}<\xi_{2,1}<\xi_{4,2}<x_A
<\xi_{3,2}<\xi_{4,3}<x_\Delta^+
<\xi_{2,2}<\xi_{4,4}<v<\xi_{3,3}<0,
\]
agreeing with all the desired results for $n\le4$. 
Let $n\ge5$. We proceed by induction.
Denote by $x_1$, $\ldots$, $x_{n-2}$ the zeros of $W_{n-2}(z)$ in increasing order. 
Then the interlacing properties by induction imply that
\[
\sgn\bg{W_{n-4}(x_j)}=\begin{cases}
(-1)^{n+j+1},&\text{if $x_j\in J_1$, or equivalently, $1\le j\in\fl{n/2}-1$};\\
(-1)^{n+j},&\text{if $x_j\in J_2$, or equivalently, $\fl{n/2}\le j\le n-3$};\\
(-1)^n,&\text{if $x_j\in J_3\cup J_4$, or equivalently, $j=n-2$}.
\end{cases}
\]
Note that $B(x_j)\ne0$.
By \cref{rec2} and the fact $W_{n-2}(x_j)=0$, 
we infer that
\[
\sgn\bg{W_n(x_j)}
=\sgn\bg{-B^2(x_j)W_{n-4}(x_j)}
=\begin{cases}
(-1)^{n+j},&\text{if $x_j\in J_1$};\\
(-1)^{n+j+1},&\text{if $x_j\in J_2$};\\
(-1)^{n+1},&\text{if $x_j\in J_3\cup J_4$}.
\end{cases}
\]
By the intermediate value theorem,
together with \cref{++++:W:xd1:xA<=xB,++++:lem:basic},
we infer that the polynomial $W_n(z)$ has zeros $z_j$ ($j\in[n]$) such that
\begin{multline*}
u
<z_1<x_1<z_2<x_2
<\cdots
<x_{\fl{n/2}-1}<z_{\fl{n/2}}<x_A\\
<z_{\fl{n/2}+1}<x_{\fl{n/2}}
<z_{\fl{n/2}+2}<x_{\fl{n/2}+1}
<\cdots<x_{n-3}<z_{n-1}<x_\Delta^+,
\end{multline*}
and 
\begin{align*}
x_\Delta^+<x_{n-2}<z_n<v,&\qquad\text{if $n$ is even},\\[4pt]
v<z_n<x_{n-2}<0,&\qquad\text{if $n$ is odd}.
\end{align*}
Since $\deg W_n(z)=n$,
we conclude that the zeros $z_j$ ($j\in[n]$)
constitute the zero set of $W_n(z)$.
This proves the real-rootedness of $W_n(z)$,
verifies the cardinalities~$|R_n^{J_j}|$,
and establishes the desired interlacing properties for $k=2$. 
In the same way, one may show those for $k=1$. 
This completes the proof. 
\end{proof}
 
The idea of piecewise interlacing also works when $x_A=x_B$,
yet some more efforts on handling the repeated zero $x_A$ are needed;
see \cref{fig:thm:++++:xA=xB} for illustration.

\begin{thm}\label{thm:++++:xA=xB}
Suppose that $x_A=x_B$.
Then the function $W_n(z)/A^{\fl{n/2}}(z)$ is a polynomial.
Denote its zero set by $S_n$.
Then 
\[
|S_{2n}^{(u,v)}|=n, \quad |S_{2n-1}^{(u,v)}|=n-1 \rmand |S_{2n-1}^{(v,0]}|=1.
\]
Moreover, for $k=1,2$, the set $S_n^{(u,v)}$
strictly interlaces $S_{n-k}^{(u,v)}$
from the left if $n$ is odd,
and from the right if~$n$ is even. 
\end{thm}

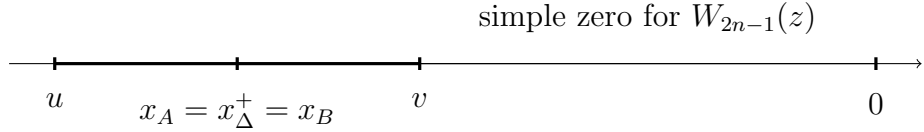
\begin{figure}[htbp]
\begin{tikzpicture}[scale=1.2]
\draw[->] (0,0) -- (10,0) coordinate (x axis);
\foreach \x in {0. 5, 2. 5, 4. 5, 9. 5} 
 \draw[very thick] (\x cm,2pt) -- (\x cm,-2pt);
\draw[very thick] (0. 5, 0) --  (4. 5, 0);

\draw (0.5, 0) node[below=7pt] {$u$};
\draw (2.5, 0) node[below=7pt] {$x_A=x_\Delta^+=x_B$};
\draw (4.5, 0) node[below=7pt] {$v$};
\draw (9. 5, 0) node[below=7pt] {$0$};

\draw (4. 5, 0) -- node[above=7pt] {simple zero for $W_{2n-1}(z)$} (9. 5, 0);
\end{tikzpicture}
\caption{Illustration of the root distribution of the polynomials $W_n(z)$ 
in \cref{thm:++++:xA=xB}.}\label{fig:thm:++++:xA=xB}
\end{figure}

\begin{proof}
Suppose that $x_A=x_B$.
Then $B(x)=c'\cdot A(x)$, where $c'=c/a$. 
Denote
\[
U_n(z)=\frac{W_n(z)}{A^{\fl{n/2}}(z)}. 
\]
In view of \cref{rec11}, the functions $U_n(z)$ 
can be defined alternatively by 
\begin{equation}\label[rec]{++++:rec:U}
U_n(z)=\begin{cases}
\displaystyle \qquad U_{n-1}(x)+c'\cdot U_{n-2}(x),&\quad\text{if $n$ is even},\\[4pt]
\displaystyle A(x)U_{n-1}(x)+c'\cdot U_{n-2}(x),&\quad\text{if $n$ is odd},
\end{cases}
\end{equation}
with $U_0(x)=1$ and $U_1(x)=x$. 
From \cref{++++:rec:U}, 
it is direct to see that every functions $U_n(z)$ is a polynomial of degree $\lceil n/2\rceil$. 
Denoted $d_n=\lceil n/2\rceil$.

By \cref{++++:lem:basic}, we deduce that 
\begin{align}
U_n(u)(-1)^{d_n}
&=\frac{W_n(u)(-1)^n}{(-A(u))^{\fl{n/2}}}(-1)^{d_n+\fl{n/2}+n}
>0,\label[ineq]{++++:sgn:U:u}\\
U_n(v)(-1)^n
&=\frac{W_n(v)(-1)^n}{A^{\fl{n/2}}(v)}>0,\rmand\label[ineq]{++++:sgn:U:v}\\
U_n(0)
&=\frac{W_n(0)}{A^{\fl{n/2}}(0)}>0. \label[ineq]{++++:sgn:U:0}
\end{align}
By \cref{rec2,++++:rec:U}, we infer that
\begin{equation}\label[rec]{++++:rec2:U}
U_n(z)=\tilde{A}(x)U_{n-2}(x)+\tilde{B}(x)U_{n-4}(x).
\end{equation}
where $\tilde{A}(x)=A(x)+2c'$ and $\tilde{B}(x)=-c'^2$.
Now, by the intermediate value theorem 
and using \cref{++++:rec:U,++++:rec2:U,++++:sgn:U:u,++++:sgn:U:v,++++:sgn:U:0},
it is routine to check that the polynomial $U_n(z)$ ($n=2,3,4$)
has zeros $\xi_{n,j}$ ($j\in[d_n]$) such that 
\[
u<\xi_{4,1}<\xi_{3,1}<\xi_{2,1}<\xi_{4,2}<v<\xi_{3,2}<0.
\]

Let $n\ge5$, 
$J_1'=(u,v)$ and $J_2'=(v,0]$. By induction, we can 
suppose that the polynomial $U_{n-2}(x)$ has zeros $x_1<x_2<\cdots<x_{d_{n-2}}$,
where 
\[
x_j\in \begin{cases}
J_2',&\text{if $n$ is odd and $j=d_{n-2}$};\\[5pt]
J_1',&\text{otherwise}.
\end{cases}
\]
Using the same technique in the proof of \cref{++++:thm:xA<xB},
one may show that 
\begin{equation}\label{pf:++++:sgn:U:xj}
\sgn\bg{U_{n}(x_j)}
=\begin{cases}
(-1)^{d_n+j},&\text{if $x_j\in J_1'$},\\[5pt]
1,&\text{if $x_j\in J_2'$},
\end{cases}
\end{equation}
and that the polynomial $U_n(z)$ has zeros $z_1$, $z_2$, $\ldots$, such that
\begin{equation}\label[ineq]{pf:++++:z}
u
<z_1<x_1<z_2<x_2
<\cdots<z_m<x_m
\end{equation}
as long as $x_m\in J_1'$.

If $n$ is even, then one may take $m=d_{n-2}=d_n-1$ in \cref{pf:++++:z}.
In this case,
since $U_n(x_{d_{n-2}})<0<U_n(v)$, 
there exists $z_{d_n}\in(x_{d_{n-2}},\,v)$.
Since $\deg U_n(z)=d_n$,
we conclude that the zeros $z_j$ ($j\in[d_n]$)
constitute the zero set of $U_n(z)$.
This proves all desired results for even $n$ with $k=2$. 
Otherwise $n$ is odd, then one may take $m=d_{n-2}-1=d_n-2$ in \cref{pf:++++:z}.
In this case, we have $x_{d_{n-2}-1}<v<x_{d_{n-2}}$.
By \cref{pf:++++:sgn:U:xj,++++:sgn:U:v}, we have 
\[
U_n(x_{d_{n-2}-1})>0,\qquad
U_n(v)<0,
\rmand
U_n(x_{d_{n-2}})>0.
\]
By the   intermediate value theorem,
the polynomial $U_n(z)$ has zeros $z_{d_n-1}$ and~$z_{d_n}$ such that
$x_{d_{n-2}-1}<z_{d_n-1}<v<z_{d_n}<x_{d_{n-2}}$.
Thus the zeros $z_j$ ($j\in[d_n]$)
constitute the zero set of $U_n(z)$.
This proves all desired results for odd $n$ and $k=2$. 

Denote by $y_1$, $\ldots$, $y_{d_{n-1}}$ 
the zeros of $U_{n-1}(x)$ in increasing order. 
In the same way, one may show that $U_n(z)$ has zeros $\mu_j$ such that
\begin{equation}\label[ineq]{pf:ym}
u<\mu_1<y_1<\mu_2<y_2<\cdots<\mu_m<y_m,
\end{equation}
as long as $y_m\in J_1'$.
If $n$ is odd, then one may take $m=d_{n-1}=d_n$ in \cref{pf:ym}.
It follows that the zero $\mu_j$ coincides with $z_j$ for $j\in[d_n]$.
Otherwise $n$ is even, one may take $m=d_{n-1}-1=d_n-2$ in \cref{pf:ym}.
Since $U_n(y_{d_{n-1}-1})<0<U_n(v)$, 
there exists $\mu_{d_n-1}\in(y_{d_{n-1}-1},\,v)$.
Since $U_n(z)$ has $d_n-1$ zeros in the interval~$J_1'$, we conclude
that the zero~$\mu_j$ coincides with $z_j$ for $j\in[d_n-1]$.
This proves all desired results for $k=1$, and completes the whole proof. 
\end{proof}

From \cref{thm:++++:xA=xB}
we see that every polynomial $W_n(z)$ has the number~$x_A$ 
as a repeated zero of multiplicity $\fl{n/2}$
or $\fl{n/2}+1$.
We remark that
by using the results in \cite{GMTW16-10},
one may find the real-rootedness of every polynomial $U_{2n}(z)$, 
the best bounds, and some convergence results of the union 
of zeros of all polynomials~$U_{2n}(z)$.

\subsection{The proof of \cref{thm:lz}.}\label[ssec]{sec:lz}

Denote by $\mathcal{L}$ the set of limits of zeros of the polynomials $W_n(z)$. 
For convenience, we define 
\[
f(z)=\bigl|A(z)+\sqrt{\Delta(z)}\bigr|-\bigl|A(z)-\sqrt{\Delta(z)}\bigr|.
\]
According to Beraha et al.'s result \cite{BKW78},
under the two non-degeneracy conditions 
\begin{itemize}
\item[(N-i)]
the sequence $\{W_n(z)\}_n$ does not satisfy a recurrence of order less than two,
\item[(N-ii)]
$f(z)\ne0$ for some $z\in\C$,
\end{itemize}
a number $z\in \mathcal{L}$ if and only if one of the following conditions is satisfied:
\begin{itemize}
\item[(C-i)]
$\alpha_-(z)=0$ and $f(z)<0$;
\item[(C-ii)]
$\alpha_+(z)=0$ and $f(z)>0$;
\item[(C-iii)]
$f(z)=0$.
\end{itemize}

In fact, Condition (N-i) is satisfied since otherwise one would have 
$W_n(z)=z^n$ for each $n$,
contradicting the fact $W_2(z)=az^2+(b+c)z+d$.
On the other hand, Condition (N-ii) holds true because
\[
f(0)=|b+\sqrt{b^2+4d}|-|b-\sqrt{b^2+4d}|\ne0.
\]
By using the maximum principle for the characteristic function of \cref{rec11},
one may show that every point satisfy Condition (C-iii) is a non-isolated limit of zeros;
see \cite[page 221]{BKW78}. In particular, 
the points $x_A$ and $x_\Delta^\pm$ are such limits 
since $f(x_A)=f(x_\Delta^\pm)=0$.

Suppose that $ad<bc$, i.e., $x_A<x_B$. 
By \cref{++++:thm:xA<xB}, we infer immediately that 
\[
\L\subseteq[u,0].
\]
Denote $J_\Delta=[x_\Delta^-,\,x_\Delta^+]$.
We claim that 
\[
J_\Delta\subseteq \L\subseteq J_\Delta\cup\{x_g^-,\,x_g^+\}.
\]
For $x\in J_\Delta$, the numbers $A(x)\pm\sqrt{\Delta(x)}$ 
are conjugate to each other.
Thus $f(x)=0$, i.e., every point in the interval $J_\Delta$ satisfies Condition (C-iii).
Hence $J_\Delta\subseteq\L$.
On the other hand,
\cref{++++:thm:xA<xB} tells us that every polynomial $W_n(z)$ has only real zeros.
Thus $\L\subseteq\R$. 
Let $x\in\L\setminus J_\Delta$.
Then either Condition (C-i) or Condition (C-ii) holds true for $z=x$.
Thus $\alpha_-(x)\alpha_+(x)=0$, which implies $g(x)=0$ from definition. 
Hence $x\in\{x_g^\pm\}$. This proves the claim.

In order to show $\L=\L^*$, we proceed according to the domain of $a$.

\begin{case}[$0<a\le 1$]\label{case:a<1}
We shall show that $\L=J_\Delta\cup\{x_g^-\}$.
From \cref{++++:lem:basic:xA<xB}, we see 
that $x_\Delta^-<x_A<x_\Delta^+<x_g^-<0$.
It follows that $A(x_g^-)>0$, $\Delta(x_g^-)>0$ and thus $f(x_g^-)>0$.
From the definition of~$g(z)$, we deduce that $\alpha_-(x_g^-)\alpha_+(x_g^-)=0$.
Since $h(x_g^-)=(2-a)x_g^--b<0$, we have 
\[
\alpha_-(x_g^-)=\frac{\sqrt{\Delta(x_g^-)}-h(x_g^-)}{2\sqrt{\Delta(x_g^-)}}>0.
\]
Hence $\alpha_+(x_g^-)=0$ and $x_g^-\in\L$ by Condition (C-ii).
On the other hand, we have $x_g^+>0$ from definition.
Since $\L\subseteq[u,0]$, we infer that $x_g^+\not\in\L$.
This proves $\L=\L^*$.
\end{case}

\begin{case}[$1<a\le 2$]\label{case:1<a<2}
We shall show that $\L=J_\Delta\cup\{x_g^+\}$.
By \cref{++++:lem:basic:xA<xB}, 
we have $x_g^-\le x_\Delta^-<x_A<x_\Delta^+<x_g^+<0$.
Note that $\L\subseteq[u,0]=[x_\Delta^-,0]$.
If $x_g^-<x_\Delta^-$, then $x_g^-\not\in\L$; 
otherwise $x_g^-=x_\Delta^-\in J_\Delta$.
On the other hand,
same to \cref{case:a<1}, we may derive
$\Delta(x_g^+)>0$, $f(x_g^+)>0$, $h(x_g^+)<0$,
$\alpha_-(x_g^+)>0$, $\alpha_+(x_g^+)=0$, and $x_g^+\in\L$. 
\end{case}

\begin{case}[$a>2$]\label{case:a>2}
We shall show that 
\[
\L=\begin{cases}
J_\Delta\cup\{x_g^+\},&\text{if $\Delta_\Delta\le \Delta_g$};\\[5pt]
J_\Delta\cup\{x_g^\pm\},&\text{if $\Delta_\Delta>\Delta_g$}.
\end{cases}
\]
In this case, 
we have $x_g^-\le x_\Delta^-<x_A<x_\Delta^+<x_g^+$ 
from \cref{++++:lem:basic:xA<xB}. 
It follows that $A(x_g^-)<0<A(x_g^+)$, $\Delta(x_g^\pm)>0$
and thus $f(x_g^-)<0<f(x_g^+)$.

Since the function $h(x)$ is decreasing and $h(x_A)=2x_A<0$, 
we have $h(x_g^+)<0$.
Along the same line as in \cref{case:1<a<2}, one may deduce that $x_g^+\in\L$.

It remains to figure out whether $x_g^-\in\L$.
If $\Delta_\Delta=\Delta_g$, then $x_g^-=x_\Delta^-\in J_\Delta$.
Below we suppose that $\Delta_\Delta\ne\Delta_g$. Then $x_g^-<x_\Delta^-$
by \cref{++++:lem:basic:xA<xB}. 
Denote by $x_h=b/(2-a)$ the zero of the function $h(z)$.
From definition, one may derive that
\[
h(x_g^-)>0 
\iff 
g(x_h)=-\frac{\Delta(x_h)}{4}>0 
\iff 
x_h<x_\Delta^+
\iff 
\Delta_\Delta>\Delta_g.
\]
If $\Delta_\Delta<\Delta_g$,
then $h(x_g^-)\le 0$ by the above equivalence relation, 
and $\alpha_-(x_g^-)>0$.
Hence $x_g^-\not\in\L$ because none of Conditions (i), (ii) and (iii) 
holds for $z=x_g^-$.
Otherwise $\Delta_\Delta>\Delta_g$.
Then $h(x_g^-)>0$ by the equivalence relation, 
and 
\[
\alpha_+(x_g^-)=\frac{\sqrt{\Delta(x_g^-)}+h(x_g^-)}{2\sqrt{\Delta(x_g^-)}}>0.
\]
It follows that $\alpha_-(x_g^-)=0$.
By Condition (i), we infer that $x_g^-\in\L$.
\end{case}
This completes the proof of \cref{thm:lz}.

\subsection{The necessity part of \cref{thm:cri:RR}.}\label[ssec]{sec:necessity}
It suffices to show that $W_n(z)$ has a non-real root for large $n$
whenever $x_A>x_B$.
Suppose that $x_A>x_B$.
Then $\Delta(x_A)=4B(x_A)>0$. 
By continuity, 
there exists $\epsilon>0$ 
such that (i) $\Delta(x)>0$ for any $x\in \mathcal{N}_\epsilon$
and (ii) $x_g^\pm\not\in \mathcal{N}_\epsilon$, where 
\[
\mathcal{N}_\epsilon=\{z\in\R\colon |z-x_A|<\epsilon\}\setminus\{x_A\}.
\]
Consequently, none of Conditions (C-i), (C-ii) and (C-iii) is satisfied by
any point $x\in \mathcal{N}_\epsilon$.
In other words, the neighbourhood $\mathcal{N}_\epsilon$ contains no limit of zeros.
By \cref{thm:lz}, the point $x_A$ is a non-isolated limit of zeros.
Therefore, there is a non-real limit of zeros in any neighbourhood of $x_A$,
which implies the existence of a non-real zero of $W_n(z)$ for every large $n$.
This completes the proof of \cref{thm:cri:RR}.

\section{Proof of \cref{thm:xA>xB:1R}}\label[sec]{sec:xA>xB}
When $x_A>x_B$, the root distribution of the polynomials $W_n(z)$ are totally
different from that when $x_A\le x_B$.
In this section, 
we show the existence of a real zero of $W_n(z)$
for large $n$ under some technical conditions.

\begin{lem}\label{++++:lem:basic:xA>xB}
Suppose that $x_A>x_B$ and $\Delta_\Delta\ge 0$.
Then 
\[
\sgn\bg{W_n(x_\Delta^+)}=\begin{cases}
(-1)^n,&\text{if $h(x_\Delta^+)\le0$ or $n<n^+$},\\[4pt]
0,&\text{if $n=n^+$},\\[4pt]
(-1)^{n+1},&\text{otherwise}.
\end{cases}
\]
Moreover, we have $h(x_\Delta^+)>0$ holds if and only if
\[
\Delta_g>\Delta_\Delta
\rmand
-ab+ac-2c>0.
\]
\end{lem}
\begin{proof}
The sign of $W_n(x_\Delta^+)$ can be determined by \cref{W:xd} directly.
Since $\Delta_\Delta\ge0$, we have $x_\Delta^\pm\in\R$.
One may compute that 
\begin{equation}\label{hxd2}
h(x_\Delta^+)=\frac{(a-2)(c-\sqrt{\Delta_\Delta})-ab}{a^2/2}.
\end{equation}
Since $x_A>x_B$, we have $\Delta_\Delta=c^2-a(ad-bc)<c^2$. 
Suppose that $h(x_\Delta^+)>0$.
In view of \cref{hxd2}, 
we infer that $a>2$ and that
\[
-ab+ac-2c=(a-2)c-ab>(a-2)\sqrt{\Delta_\Delta}
\]
is positive. Squaring both sides of the above inequality gives
\[
0<(-ab+ac-2c)^2-(a-2)^2\Delta_\Delta=a^2(\Delta_g-\Delta_\Delta).
\]
This proves the sufficiency. Conversely, the inequality $-ab+ac-2c>0$ implies $a>2$.
The above deductions holds true line by line backwards, as desired.
\end{proof}

\begin{thm}\label{thm:hxd+>0}
Suppose that $x_A>x_B$ and $h(x_\Delta^+)>0$.
Let $n\ge n^+$. Then 
the polynomial $W_n(z)$ has at least two distinct real zeros.
\end{thm}

\begin{proof}
By \cref{++++:lem:basic:xA>xB}, we have $a>2$ and the polynomial $h(z)$
is linear.
Since $h(x_\Delta^+)>0$ is real,
we deduce that $x_\Delta^+\in\R$, that is, $\Delta_\Delta\ge 0$.
By \cref{++++:lem:basic:xA>xB}, we have $\Delta_g>\Delta_\Delta\ge0$.
It follows that $x_\Delta^\pm,x_g^\pm\in\R$.
Applying Vi\`eta's theorem to the quadratic polynomial $g(z)$, 
one infers that $x_g^\pm<0$.
Since $g(x_\Delta^\pm)=h^2(x_\Delta^\pm)/4\ge0$, 
we find 
\begin{equation}\label[ineq]{ineq:xgxd}
x_g^-\le x_\Delta^-\le x_\Delta^+\le x_g^+<0.
\end{equation}
Since $a>2$, we infer that 
\[
h(x_\Delta^-)=\frac{(a-2)(c+\sqrt{\Delta_\Delta})-ab}{a^2/2}
\ge\frac{(a-2)(c-\sqrt{\Delta_\Delta})-ab}{a^2/2}
=h(x_\Delta^+)>0.
\]
It follows that the difference
\begin{equation}\label{n+-n-}
n^+-n^-
=-\frac{A(x_\Delta^+)}{h(x_\Delta^+)}+\frac{A(x_\Delta^-)}{h(x_\Delta^-)}
=\frac{8b\sqrt{\Delta_\Delta}}{a^2h(x_\Delta^-)h(x_\Delta^+)}
\end{equation}
is positive.
Let $n\ge n^+$. 
From \cref{++++:lem:basic,++++:lem:basic:xA>xB,W:xg}, 
we see that
\[
W_n(x_\Delta^-)(-1)^{n+1}>0,\quad
W_n(x_\Delta^+)(-1)^{n+1}\ge 0,
\rmand
W_n(x_g^\pm)(-1)^n>0.
\]
As a consequence, we have $x_g^-\ne x_\Delta^-$ and $x_\Delta^+\ne x_g^+$.
By using the intermediate value theorem, 
we find that every polynomial $W_n(z)$ has a zero 
in the interval $(x_g^-,\,x_\Delta^-)$ 
and another zero in the interval $[\,x_\Delta^+,\,x_g^+)$.
\end{proof}

\begin{thm}\label{thm:hxd->0}
Suppose that $x_A>x_B$ and $h(x_\Delta^+)<0<h(x_\Delta^-)$.
Let $n>n^-$.
Then the polynomial $W_n(z)$ has at least one real zero.
If $\Delta_g\ge0$ in addition, 
then~$W_n(z)$ has at least two distinct real zeros.
\end{thm}

\begin{proof}
From the premise $h(x_\Delta^-)>0$, we infer that $a>2$.
Along the lines of the proof of \cref{thm:hxd+>0},
one may obtain \cref{ineq:xgxd}, and $n^->n^+$ by \cref{n+-n-}.
Let $n>n^-$. 
By \cref{W:xg,++++:lem:basic,++++:lem:basic:xA>xB}, 
we have 
\[
W_n(x_g^-)(-1)^n>0,\quad
W_n(x_\Delta^-)(-1)^{n+1}>0
\rmand
W_n(x_\Delta^+)(-1)^n>0.
\]
Thus $x_g^-\ne x_\Delta^-$ and $x_\Delta^-\ne x_\Delta^+$.
By the intermediate value theorem, 
the polynomial~$W_n(z)$ has a real zero in the interval $(x_\Delta^-,\,x_\Delta^+)$.
If, additionally, $\Delta_g\ge 0$, then $x_g^-\in\R$,
and the polynomial~$W_n(z)$ has another real zero in the interval 
$(x_g^-,\,x_\Delta^-)$.
This completes the proof.
\end{proof}

Combining \cref{thm:hxd+>0,thm:hxd->0}, one obtains a complete proof of \cref{thm:xA>xB:1R},
in which the number $N$ can be taken to be $\lceil\max(n^-,\,n^+,\,1)\rceil$. In fact, with some extra efforts
one may show that $n^+>1$ in \cref{thm:hxd+>0} and that $n^->1$ in \cref{thm:hxd->0}. Hence the number
$N$ can be set to $\lceil\max(n^-,\,n^+)\rceil$.

\end{document}